\documentclass[journal]{IEEEtran}

\usepackage[noadjust]{cite}
\usepackage{algorithmic}
\usepackage{graphicx}
\usepackage{textcomp}

\usepackage{amsmath,amsthm,amssymb,amsfonts,amsbsy}
\usepackage{mathtools}
\usepackage{epsfig,epsf,url,enumitem}
\usepackage{accents}

\usepackage{chngcntr}
\usepackage{apptools}
\AtAppendix{\counterwithin{thm}{subsection}}
\AtAppendix{\counterwithin{lemma}{subsection}}
\newcommand{\bmat}[1]{\begin{bmatrix}#1\end{bmatrix}}

\newcommand{\Prob}[1]{\mathbb{P} \left( #1 \right)}
\newcommand{\Exp}[1]{\mathbb{E} \left[ #1 \right]}
\newcommand{\Expx}[1]{\mathbb{E}_{x_0 \sim \mathcal{D}} \left[ #1 \right]}
\newcommand{\Expxt}[1]{\mathbb{E}_{x_0 \sim \mathcal{D}, \omega_0\sim \pi} \left[ #1 \right]}

\newcommand{\1}[1]{\mathbf{1}_{#1}}

\newcommand{\norm}[2]{\left\| #1 \right\|_{#2}}
\newcommand{\inner}[2]{\langle #1,\, #2 \rangle}
\newcommand{\R}{\mathbb{R}}

\newcommand{\tr}[1]{\,\textup{tr} \left( #1 \right)}
\newcommand{\diag}[1]{\textup{diag}\! \left( #1 \right)}
\newcommand{\vect}[1]{\textup{vec}\! \left( #1 \right)}
\newcommand{\srad}[1]{\rho ( #1 )}

\newcommand{\sumtinf}{\sum_{t=0}^\infty}
\newcommand{\sumomega}{\sum_{i\in\Omega}}

\newcommand{\deltaKi}{\Delta K_i}

\newcommand{\X}{\mathbf{X}}
\newcommand{\K}{\mathbb{K}}

\newcommand{\M}{\mathbb{M}}

\newcommand{\Pmat}{\mathcal{P}}

\newcommand{\Lc}[1]{\mathcal{L}( #1 )}
\newcommand{\Li}[1]{\mathcal{L}_i( #1 )}
\newcommand{\T}[1]{\mathcal{T}( #1 )}
\newcommand{\Tj}[1]{\mathcal{T}_j( #1 )}
\newcommand{\E}[1]{\mathcal{E}( #1 )}
\newcommand{\Ei}[1]{\mathcal{E}_i( #1 )}

\newcommand{\Lt}[1]{\mathcal{L}^t( #1 )}
\newcommand{\Tt}[1]{\mathcal{T}^t( #1 )}

\newcommand{\cinv}{-1}
\newcommand{\cT}{T}
\newcommand{\mult}{}

\newtheorem{thm}{Theorem}
\newtheorem{lemma}{Lemma}

\newtheorem{prop}{Proposition}
\newtheorem*{myprb}{Problem: Policy Optimization for MJLS}

\theoremstyle{definition}
\newtheorem{remark}{Remark}

\makeatletter
\let\NAT@parse\undefined
\makeatother
\usepackage{hyperref}

\hypersetup{pdfauthor={Joao Paulo Jansch-Porto},%
            pdftitle={Policy Optimization for Markovian Jump Linear Quadratic Control: Gradient-Based Methods and Global Convergence},%
            pdfsubject={MJLS; Switched Systems; Reinforcement Learning; Policy Gradient},%
            pdfkeywords={LQR; hybrid systems; MJLS; switched systems; Reinforcement Learning; Policy Gradient},%
            colorlinks  = true, 
            urlcolor    = blue, 
            linkcolor   = blue, 
            citecolor   = red, 
}

\title{\LARGE \bf
 Policy Optimization for Markovian Jump Linear Quadratic Control:\\
 Gradient-Based Methods and Global Convergence
}

\author{
Joao Paulo Jansch-Porto, Bin Hu, and
    Geir E. Dullerud\thanks{J. P. Jansch-Porto is with the Department of Mechanical Science and Engineering,
  University of Illinois at Urbana-Champaign, Email:   \texttt{janschp2@illinois.edu}.}\thanks{B.~Hu is with the  Coordinated Science Laboratory (CSL) and the  Department of Electrical and Computer Engineering,
  University of Illinois at Urbana-Champaign, Email:   \texttt{binhu7@illinois.edu}.}\thanks{G. E. Dullerud is with the  Coordinated Science Laboratory (CSL) and the Department of Mechanical Science and Engineering,
  University of Illinois at Urbana-Champaign, Email:   \texttt{dullerud@illinois.edu}.} 
}

\begin{document}

\maketitle

\begin{abstract}
Recently, policy optimization for control purposes has received renewed attention due to the increasing interest in reinforcement learning. 
In this paper, we investigate the global convergence of gradient-based policy optimization methods for quadratic optimal control of discrete-time Markovian jump linear systems (MJLS). 
First, we study the optimization landscape of direct policy optimization for MJLS, with static state feedback controllers and quadratic performance costs.
Despite the non-convexity of the resultant problem, we are still able to identify several useful properties such as coercivity, gradient dominance, and almost smoothness.
Based on these properties, we show global convergence of three types of policy optimization methods: the gradient descent method; the Gauss-Newton method; and the natural policy gradient method.
We prove that all three methods converge to the optimal state feedback controller for MJLS at a linear rate if initialized at a controller which is mean-square stabilizing. 
Some numerical examples are presented to support the theory.
This work brings new insights for understanding the performance of policy gradient methods on the Markovian jump linear quadratic control problem.

\end{abstract}


\section{Introduction}
\label{sec:intro}

Recently, reinforcement learning (RL)~\cite{sutton2018reinforcement} has achieved impressive performance on a class of continuous control problems including locomotion~\cite{schulman2015high} and robot manipulation~\cite{levine2016end}. Policy optimization is the main engine behind these RL applications~\cite{duan2016benchmarking}. Specifically, the policy gradient method~\cite{sutton2000policy} and several related methods including natural policy gradient~\cite{kakade2002natural}, TRPO~\cite{schulman2015trust}, natural AC \cite{peters2008natural}, and PPO~\cite{schulman2017proximal} are among the most popular RL algorithms for continuous control tasks. 
These methods enable flexible policy parameterizations and optimize control performance metrics directly.

Despite the empirical successes of policy optimization methods, how to choose these algorithms for a specific control task is still more of an art than a science~\cite{henderson2018deep, rajeswaran2017towards}.
This motivates a recent research trend focusing on understanding the performances of policy optimization algorithms on simplified benchmarks such as  linear quadratic regulator (LQR)~\cite{pmlr-v80-fazel18a,bu2019lqr,malik2018derivative,tu2018gap,yang2019global,Krauth2019,mohammadi2019convergence,mohammadi2019global,fatkhullin2020optimizing,furieri2020learning,li2019distributed}, linear robust control~\cite{zhang2019policyc,gravell2019learning,zhang2020stability}, and linear control of Lur'e systems~\cite{qu2020combining}.
Notice that even for LQR, directly optimizing over the policy space leads to a non-convex constrained problem.
Nevertheless, one can still prove the global convergence of  policy gradient methods on the LQR problem by exploiting properties such as gradient dominance, almost smoothness, and coercivity~\cite{pmlr-v80-fazel18a,bu2019lqr}. 
This provides a good sanity check for applying policy optimization to more advanced control applications.

Built upon the good progress on understanding policy-based RL for the LQR problem in linear time-invariant systems, this paper moves one step further and studies policy optimization for Markov jump linear systems (MJLS) \cite{costa2006discrete} from a theoretical perspective.
MJLS form an important class of hybrid dynamical systems that find many applications in control~\cite{bar1993estimation,fox2011tsp,hamsa2016cdc, Pavlovic2000LearningSL,sworder1999estimation,varga2013ijrnc} and machine learning~\cite{hu2017unified,hu2019characterizing}.
The research on MJLS has great practical value while in the mean time  also provides many new interesting theoretical questions. 
In the classic LQR problem, one aims to control a linear time-invariant (LTI) system whose state/input matrices do not change over time. On the other hand, the state/input matrices of a Markov jump linear system are functions of a jump parameter that is sampled from an underlying Markov chain. Consequently,
the behaviors of MJLS become very different from  those of LTI systems.  
Controlling unknown MJLS poses many new challenges over  traditional LQR due to the appearance of this Markov jump parameter, and it is the coupling effect between the state/input matrices and the jump parameter distribution causes the main difficulty.
To this end, the optimal control of MJLS provides a meaningful benchmark for further understanding of policy-based RL algorithms.

However, the theoretical properties of policy-based RL methods on discrete-time MJLS have been overlooked in the existing literature~\cite{costa2002monte, beirigo2018online,9029946,vargas2015gradient}.
In this paper, we make an initial step towards bridging this gap. 
Specifically, we develop new convergence theory for direct policy optimization on the quadratic control problem of MJLS. 
First, we study the policy optimization landscape in the MJLS setting.
Despite the non-convexity of the resultant problem, we are still able to identify several useful properties such as coercivity, gradient dominance, and almost smoothness.
Next, we use these identified properties to show global convergence of three types of policy optimization methods: the gradient descent method, the Gauss-Newton method, and the natural policy gradient method.
We prove that all these methods converge to the optimal state feedback controller for MJLS at a linear rate if a stabilizing initial controller is used.  Finally, numerical results are provided to support our theory.
 
This paper expands on the initial results published by the authors in a conference paper~\cite{joaoACC}, and has made significant extensions in analyzing the gradient descent method in the MJLS setting. 
Our work serves as an initial step toward understanding the theoretical aspects of policy-based RL methods for MJLS control.

\section{Background and Problem Formulation}
\label{sec:background}

\subsection{Notation}\label{sec:notation}
We denote the set of real numbers by \(\R\).
Let \(A\) be a matrix, then we use the notation \(A^T\), \( \|A\| \),  \(\tr{A}\), \(\sigma_{\min}(A)\), and \(\rho(A) \) to denote its transpose, maximal singular value, trace, minimum singular value, and spectral radius, respectively. 
Given matrices \(\{D_i\}_{i = 1}^m\),  let \( \diag{D_1, \ldots, D_m}\)  denote the  block diagonal matrix whose $(i,i)$-th block is $D_i$. 
The Kronecker product  of matrices $A$ and $B$ is denoted as $A\otimes B$. We use
$\textup{vec}(A)$ to denote the vectorization of matrix $A$.
We indicate when a symmetric matrix $Z$ is positive definite or positive semidefinite matrices by \(Z\succ 0\) and \(Z \succeq 0\), respectively.
Given a function $f$, we use $df$ to denote its total derivative~\cite{abraham2012manifolds}.


We now introduce some specific matrix spaces and  notation motivated from the MJLS literature~\cite{costa2006discrete}.
Let $\M^N_{n\times m}$ denote the space made up of all $N$-tuples of real matrices $V = (V_1, \ldots, V_N)$ with $V_i \in \R^{n\times m}, i \in \mathbb{N}$.
For simplicity, we write $\M^N$ in place of $\M^N_{n\times m}$ when the dimensions $n$ and $m$ are clear from context.
For $V = (V_1, \ldots, V_N) \in \M^N$, we define
\begin{align*}
&\|V\|_1 \coloneqq \sumomega \|V_i\|, \quad \|V\|_2^2 \coloneqq \sum_{i = 1}^{N} \tr{V_i^T V_i}, \\
&\|V\|_{\max} \coloneqq \max_{i= 1,\ldots,N}\|V_i\|, \quad \Lambda_{\min}(V) \coloneqq \min_{i= 1,\ldots,N} \sigma_{min}(V_i).
\end{align*}
Clearly, we have
$\|V\|_{\max} \leq \|V\|_1 \leq \|V\|_2$.
For $V, S \in \M^N$, their inner product is defined as 
\begin{equation*}
\inner{V}{S} \coloneqq \sum_{i = 1}^{N} \tr{V_i^T S_i}
\end{equation*}
Notice both $V$ and $S$ are sequences of matrices.
It is also convenient to define $V+S \coloneqq (V_1+ S_1, \ldots, V_N+ S_N)$, $V\mult S \coloneqq (V_1 S_1, \ldots, V_N S_N)$, $V^{\cT} \coloneqq (V_1^{T}, \ldots, V_N^{T})$, and $V^{\cinv} \coloneqq (V_1^{-1}, \ldots, V_N^{-1})$.
We say that $V \succ S$ if $V_i - S_i \succ 0 $ for $i = 1, \ldots, N$. 

\subsection{Markovian Jump Linear Quadratic  Control}
A Markovian jump linear system (MJLS) is governed by the following discrete-time state-space model
\begin{equation} \label{eq:ltv}
x_{t+1} = A_{\omega(t)} x_t + B_{\omega(t)} u_t
\end{equation}
where \( x_t \in \R^d\) is the system state, and \(u_t \in \R^k \) corresponds to the control action.
The initial state \(x_0\) is assumed to have a distribution \(\mathcal{D}\).
The system matrices \( A_{\omega(t)} \in \R^{d\times d}\) and \(B_{\omega(t)} \in \R^{d\times k} \) depend on the switching parameter $\omega(t)$, which takes values on \( \Omega\coloneqq\{1, \ldots, N_s \} \). 
We will denote \(A = (A_1, \ldots, A_{N_s}) \in \M_{d\times d}^{N_s}\) and \(B = (B_1, \ldots, B_{N_s}) \in \M_{d\times k}^{N_s}\).

The jump parameter \(\{\omega(t)\}_{t=0}^\infty\) is assumed to form a time-homogeneous Markov chain whose transition probability is given as
\begin{equation}\label{eq:prob}
p_{ij} = \Prob{\omega(t+1) = j | \omega(t) = i}.
\end{equation}
Let $\Pmat$ denote the probability transition matrix whose $(i,j)$-th entry is $p_{ij}$.
The initial distribution of $\omega(0)$ is given by $\pi = \bmat{\pi_1\! &\! \!\cdots\!\! & \!\pi_{N_s}}^T$.
Obviously, we have \(p_{ij} \geq 0\), \(\sum_{j=1}^{N_s} p_{ij} =1\), and \( \sum_{i\in\Omega} \pi_i = 1 \).  
We further assume that system~\eqref{eq:ltv} 
is mean-square stabilizable\footnote{The mean square stability of MJLS is reviewed in sequel.}.

In this paper, we focus on the quadratic optimal control problem whose objective is to choose the control actions \(\{u_t\}_{t=0}^\infty\) to minimize the following cost function
\begin{equation} \label{eq:switched_cost}
C = \Expxt{\sumtinf x_t^T Q_{\omega(t)} x_t + u_t^T R_{\omega(t)} u_t}.
\end{equation}
For simplicity,  it is assumed that \(Q = (Q_1, \ldots, Q_{N_s})\succ 0\), \(R = (R_1, \ldots, R_{N_s})\succ 0\),  $\pi_i > 0$, and  \(\Expx{x_0 x_0^T} \succ 0\).
The assumptions on $\pi$ and $\Expx{x_0 x_0^T}$ indicate that there is a chance of starting from any mode $i$ and the covariance of the initial state is full rank.
These assumptions can be somehow informally thought as the persistently excitation condition in the system identification literature and are quite standard for learning-based control.
The above problem can be viewed as the MJLS counterpart of the standard LQR problem, and hence is termed as the ``MJLS LQR problem."
The optimal controller to this MJLS LQR problem, defined by dynamics~\eqref{eq:ltv}, cost~\eqref{eq:switched_cost}, and switching probabilities~\eqref{eq:prob}, can be computed by solving a system of coupled  Algebraic Riccati Equations (AREs)~\cite{fragoso}.
Notice that it is known that the optimal cost can be achieved by a linear state feedback of the form
\begin{equation}\label{eq:control}
u_t = -K_{\omega(t)} x_t
\end{equation}
with \(  K= (K_1, \ldots, K_{N_s}) \in \M_{k \times d}^{N_s} \).
Combining the linear policy~\eqref{eq:control} with~\eqref{eq:ltv}, we obtain the closed-loop dynamics:
\begin{equation}\label{eq:cl}
x_{t+1} = \left(A_{\omega(t)} - B_{\omega(t)}K_{\omega(t)} \right)x_t  = \varGamma_{\omega(t)} x_t.
\end{equation}
with $\varGamma = (\varGamma_1, \ldots, \varGamma_{N_s})\in\M_{d\times d}^{N_s}$.
Note that using this formulation, we can write the cost~\eqref{eq:switched_cost} as
\begin{equation*}
C = \Expxt{\sumtinf x_t^T \left( Q_{\omega(t)} + K_{\omega(t)}^T R_{\omega(t)} K_{\omega(t)}  \right) x_t}.
\end{equation*}

Now we briefly review how to solve the above MJLS LQR problem.
First, we define the operator $\mathcal{E}:\M_{d\times d}^{N_s}\rightarrow \M_{d\times d}^{N_s}$ as $\E{V} \coloneqq (\mathcal{E}_1(V), \ldots, \mathcal{E}_{N_s}(V))$ where $V=(V_1,\ldots, V_{N_s})\in\M_{d\times d}^{N_s}$ and $\Ei{V} \coloneqq \sum_{j = 1}^{N_s} p_{ij} V_j$.
Let \( P = (P_1, \ldots, P_{N_s}) \) be the unique positive definite solution to the following  AREs:
\begin{align}\label{eq:markov_riccati}
P &= Q +  A^{\cT} \mult \E{P} \mult A - A^{\cT} \mult \E{P} \mult B \times \nonumber\\
&\qquad\qquad\qquad\qquad\; \left( R + B^{\cT}\mult  \E{P}\mult  B \right)^{\cinv}\mult  B^{\cT} \mult \E{P}\mult  A.
\end{align}
It can be shown that the optimal controller is given by
\begin{equation}\label{eq:opt_k}
K^* =\left( R + B^{\cT}\mult  \E{P} \mult B \right)^{\cinv} B^{\cT} \mult \E{P} \mult  A.
\end{equation}

Notice that 
the existence of such a controller is guaranteed by the stabilizability assumption.
In this paper, we will revisit the above MJLS LQR problem from a policy optimization perspective.

\begin{remark}
If \({\omega(t)} = {\omega(t+N_s)}\) for all \(t\), then the system is said to be periodic with period \( N_s \).
Linear periodic systems have been widely studied~\cite{Bittanti1991,293179} and are just a special case of MJLS.
If \(N_s = 1\), then the MJLS just becomes a linear time-invariant (LTI) system.
\end{remark}

\subsection{Policy Optimization for LTI Systems}
\label{sec:LQRreview}
Before proceeding to policy optimization of MJLS, here we review policy gradient methods for the quadratic control of  LTI systems~\cite{pmlr-v80-fazel18a}. 
Consider the LTI system $x_{t+1}=A x_t+B u_t$, $A\in\R^{d\times d}$, $B\in\R^{d\times k}$, with an initial state distribution $\mathcal{D}$ and a static state feedback controller $u_t=-Kx_t$.
We adopt a standard quadratic cost function which can be calculated as
\begin{align} \label{eq:lti_cost}
C(K) &= \Expx{ \sum_{t=0}^\infty x_t^T Q x_t + u_t^T R u_t }\nonumber \\
&=\Expx{ \sum_{t=0}^\infty x_t^T (Q+K^T R K) x_t}.
\end{align}
Obviously, the cost in \eqref{eq:lti_cost} can be computed as
$C(K) = \Expx{x_0^T P^{K} x_0}$ where  $P^K$ is the solution to the Lyapunov equation
$P^K = Q + K^T R K + (A - BK)^T P^K (A-BK)$.
It is also well known~\cite{maartensson2009gradient, pmlr-v80-fazel18a} that the gradient of~\eqref{eq:lti_cost} with respect to \(K\)  can be calculated as
\begin{equation*}
\nabla C(K) = 2 \left(  \left( R + B^T P^K B\right) K - B^T P^K A \right) \Sigma_K,
\end{equation*}
where $\Sigma_K$ is the state correlation matrix, i.e. $\Sigma_K  = \Expx{\sum_{t=0}^\infty x_t x_t^T}$. 
Based on this gradient formula, one can optimize \eqref{eq:lti_cost} using the policy gradient method $K'\leftarrow K-\eta \nabla C(K)$,  the Gauss-Newton method $K'\leftarrow K-\eta (R+B^T P^K B)^{-1} \nabla C(K)\Sigma_K^{-1}$, or the natural policy gradient method $K'\leftarrow \bar{K}-\eta \nabla C(K) \Sigma_K^{-1}$.
One advantage of these gradient-based methods is that they can be implemented in a model-free manner. More explanations for these  methods can be found in~\cite{pmlr-v80-fazel18a}.

In \cite{pmlr-v80-fazel18a}, it is shown that there exists a unique $K^*$ such that $\nabla C(K^*)=0$ if $\Expx{ x_0 x_0^T}$ is full rank.  
In addition, all the above methods are shown to converge to $K^*$ linearly if a stabilizing initial policy is used.

\subsection{Problem Setup: Policy Optimization for MJLS}
In this section, we reformulate the MJLS LQR problem as a policy optimization problem. Since we know the optimal cost for the MJLS LQR problem can be achieved by a linear state feedback controller, it is reasonable to  restrict the policy search within the class of linear state feedback policies. Specifically, we can set $K = (K_1, \ldots, K_{N_s})$,
where each of the components is the feedback gain of
the corresponding mode. With this notation, we consider
the following policy optimization problem whose decision
variable is $K$.

\begin{myprb}
\begin{align*}
\text{minimize:} &\quad \text{cost } C( K),\text{ given in~\eqref{eq:switched_cost}}\\
\null \text{subject to:} &\quad \text{state dynamics, given in~\eqref{eq:ltv}} \\
        &\quad \text{control actions, given in~\eqref{eq:control}}\\
        &\quad \text{transition probabilities, given in~\eqref{eq:prob}}\\
   & \quad \text{stability constraint, } K \text{ stabilizing \eqref{eq:ltv} in the}\\
   & \qquad \text{mean square sense.}
\end{align*}
\end{myprb}

When $N_s=1$, the above problem reduces to the policy optimization for LTI systems \cite{pmlr-v80-fazel18a}.
We want to emphasize that the above problem is indeed a constrained optimization problem. 
Recall that given $K$, the resultant closed-loop MJLS~\eqref{eq:cl} is mean square stable (MSS) if for any initial condition $x_0\in\R^d$ and $\omega(0)\in\Omega$, one has
$\Exp{x_t x_t^T} \rightarrow 0$ as $t\rightarrow \infty$~\cite{costa2006discrete}. 
Since it is assumed \(\Expx{x_0 x_0^T} \succ 0\), we can trivially apply the well-known equivalence between mean square stability and stochastic stability for MJLS~\cite{costa2006discrete} to show that $C(K)$ is finite if and only if $K$ stabilizes the closed-loop dynamics in the mean square sense.
Therefore,
the feasible set of the above policy optimization problem consists of all $K$ stabilizing the closed-loop dynamics \eqref{eq:cl} in the mean square sense. For simplicity, we denote this feasible set as $\K$.
For $K \in \K$, $C(K)$ can be calculated as
\begin{align} \label{eq:markov_cost}
C(K) &=  \Expxt{x_0^T P_{\omega(0)}^{K} x_0}\nonumber \\
&= \Expx{x_0^T \left( \sum_{i\in\Omega } \pi_i P_i^{{K}} \right) x_0},
\end{align}
where $P^K = (P^K_1, \ldots, P^K_{N_s}) \in\M_{d\times d}^{N_s}$ and each $P_i^K$ is solved via the following coupled Lyapunov equations:
\begin{equation}\label{eq:lyap_markov}
P_i^K = Q_i + K_i^T R_i K_i + \left( A_i - B_i K_i \right)^T \Ei{P^K} \left( A_i - B_i K_i \right).
\end{equation}

The goal for policy optimization is to apply iterative gradient-based methods to search for the cost-minimizing element $K^*$ within the feasible set $\K$.
A fundamental question is how to check whether $K\in \K$ for any given $K$.
There are several ways to do this, and we give a brief review here.
We need to introduce a few operators which are standard in the MJLS literature.
Specifically, for any $V\in\M_{d\times d}^{N_s}$, we define 
$\T{V} = (\mathcal{T}_1(V), \ldots, \mathcal{T}_{N_s}(V)) \in \M_{d\times d}^{N_s}$, where $\Tj{V}$ is computed as 
\begin{align*}
\Tj{V} &\coloneqq \sumomega p_{ij} (A_i - B_i K_i) V_i (A_i - B_i K_i)^T.
\end{align*}
Recall that  $\Ei{V} \coloneqq \sum_{j = 1}^{N_s} p_{ij} V_j$.
We can also define $\Lc{V} = (\mathcal{L}_1(V), \ldots, \mathcal{L}_{N_s}(V))\in \M_{d \times d}^{N_s}$, where $\Li{V}$ is given as
\begin{align*}
\Li{V} &\coloneqq (A_i - B_i K_i)^T \Ei{V} (A_i - B_i K_i).
\end{align*}
The following property of $\mathcal{E}_i$ is quite useful
\begin{align}\label{eq:Ebound}
\|\Ei{V}\|  \leq \sum_{j\in\Omega} p_{ij} \|V_i\| \leq \|V\|_{\max} \left( \sum_{j\in\Omega} p_{ij}\! \right)=\|V\|_{\max}.
\end{align}
It is also easy to check that both $\mathcal{T}$ and $\mathcal{L}$ are Hermitian and positive operators.
From~\cite{costa2006discrete}, we also know $\mathcal{T}$ is the adjoint operator of $\mathcal{L}$. The operator $\mathcal{T}$ is useful in describing the covariance propagation of the MJLS \eqref{eq:cl}. Specifically, if we define $X(t)=\left(X_1(t), \ldots, X_{N_s}(t)\right)$ with \(X_i(t) \coloneqq \Exp{x_t x_t^T \1{ \omega(t) = i}}\), then we have
 $X(t+1) = \T{X(t)}$.
 In addition, we know $\sum_{t=0}^\infty X(t)$ exists if $K\in \K$. We denote this limit as $\X^K$ and we have  
 \begin{align}\label{eq:XK}
 \X^K = \sum_{t = 0}^{\infty} \Tt{X(0)}.
 \end{align}
 The operator $\mathcal{L}$ is useful for value computation, since we have
 $P^{{K}} = \Lc{P^{{K}}} + Q + K^{\cT} \mult  R \mult K$ (or equivalently $P^{K} = \sum_{t = 0}^{\infty} \Lt{Q+K^{\cT} \mult  R\mult  K}$) for any $K\in \K$. Also notice $\mathcal{L}$ is actually a linear operator and has a matrix representation $\mathcal{A}\coloneqq\diag{\Gamma_i^T\otimes \Gamma_i^T}(\Pmat \otimes I_{N_s^2})$ where $\Gamma_i=A_i-B_i K_i$ (see Proposition 3.4 in \cite{costa2006discrete} for more details). 
 Now we are ready to present the following well-known result which can be used to check whether $K$ is in $\K$ or not.


\begin{prop}[\cite{costa2006discrete}]\label{prop:mss}
The following assertions are equivalent:
\begin{enumerate}
    \item System~\eqref{eq:cl} is MSS.
    \item $\rho(\mathcal{A}) < 1$.
    \item For any $S\in\M_{n\times n}^{N_s}$, $S\succ 0$, there exists a unique $V\in \M_{n\times n}^{N_s}$, $V\succ0$, such that $V - \T{V} = S$.
    \item There exists $V\succ 0\in \M_{n\times n}^{N_s}$ such that  $V - \T{V} \succ 0$.
\end{enumerate}
The results above also hold when replacing $\mathcal{T}$ by $\mathcal{L}$.
\end{prop}

Based on the above result, a few basic properties of $\K$ can be obtained.
Clearly, we have $\K\coloneqq \{ K \in \M_{k\times d}^{N_s}: \rho(\mathcal{A}) < 1  \}$. Since $\rho(\mathcal{A})$ is a continuous function of $K$, we know $\K$ is an open set and $\K^c$ is a closed set.
The boundary of the set $\K$ can also be formally specified as $\partial \K\coloneqq \{ K\in \M_{k\times d}^{N_s} : \rho(\mathcal{A}) = 1  \}$.

Finally, it is worth mentioning that both $\mathcal{L}$ and $\mathcal{T}$ depend on $K$. Occasionally, we will use the notation $\mathcal{L}^K$ and $\mathcal{T}^K$ when there is a need to emphasize the dependence of these operators on $K$.

\section{Optimization Landscape and Cost Properties}
\label{sec:pg}
In this section, we study the optimization landscape of the MJLS LQR problem and identify several useful properties of $C(K)$.
First, we present an explicit formula for the policy gradient \(\nabla C(K)\).
Notice that $K$ is a tuple of real matrices and hence we have $\nabla C(K) \in \M_{k\times d}^{N_s}$. 

\begin{lemma} \label{lemma:policy_grad}
Suppose $K\in \K$.
Then the cost ~\eqref{eq:markov_cost} is continuously differentiable with respect to ${K}$, and
 the gradient $\nabla C({K})$ can be calculated as
\begin{equation}\label{eq:exact_grad}
\nabla C({K}) = 2 L^K \mult \X^K
\end{equation}
where $L^K = (L_1^K, \ldots, L_{N_s}^K) \in \M_{k\times d}^{N_s}$ and $L_i^K$ is given by
\begin{equation}
\label{eq:Ldef}
L_i^K = \left(R_i + B^T_i \Ei{P^K} B_i \right) K_i - B^T_i \Ei{P^K} A_i.
\end{equation}
Moreover, $\X^K$ in the above gradient formula is given by \eqref{eq:XK}.
\end{lemma}
\begin{proof}
The differentiability of $ C( K)$  can be proved using the implicit function theorem, and this step is similar to the proof of Lemma 3.1 in \cite{rautert1997computational}. 
Specifically, for every $K\in\K$, let $P^K$ be the solution to the equation
\begin{equation*}
P^K = \Lc{P^K} + Q + K^{\cT} \mult R\mult  K
\end{equation*}
Recall that for any $V \in \M_{n\times n}^{N_s}$ we have $\vect{\Lc{V}} = \mathcal{A}\, \vect{V}$ (see Proposition 3.4 in \cite{costa2006discrete}).
Then
\begin{align*}
\vect{P^K} &= \vect{\Lc{P^K}} + \vect{Q + K^{\cT}\mult R\mult K} \\
 &= \mathcal{A}\, \vect{P^K} + \vect{Q + K^{\cT} \mult R \mult K}
\end{align*}
Since $K \in \K$, we know $\srad{\mathcal{A}} < 1$, which in turn implies that $(I - \mathcal{A})$ is invertible.
By the implicit function theorem, the map $K\mapsto P^K$ is continuous, and the function $C({K})$ is continuously differentiable with respect to $K$.

Now we derive the gradient formula by modifying the total derivative arguments in~\cite{rautert1997computational, maartensson2009gradient}.
We can take the total derivative of~\eqref{eq:lyap_markov} to show the following relation for each $i \in \Omega$ 
\begin{align*}
dP_i^{K} &= dK_i^T L_i^K + (L_i^K)^T dK_i + \varGamma_i^T  \left( \sum_{j\in\Omega} p_{ij} dP_j^{{K}} \right) \varGamma_i\\
	&= dK_i^T L_i^K + (L_i^K)^T dK_i + \varGamma_i^T \Ei{dP^K} \varGamma_i.
\end{align*}
Therefore, the above equation can be compactly rewritten as
\begin{align*}
    dP^K=dK^{\cT}\mult L^K + (L^K)^{\cT} \mult dK+\Lc{dP^K}
\end{align*}
which is equivalent to 
\begin{align*}
dP^K=\sum_{t=0}^\infty \mathcal{L}^t(dK^{\cT}\mult L^K + (L^K)^{\cT} \mult dK).
\end{align*}
Notice that we can rewrite the cost~\eqref{eq:markov_cost} as
\begin{equation}\label{eq:cost_inner}
C(K) = \inner{P^K}{X(0)}
\end{equation}
Therefore, we can take total derivative of \eqref{eq:cost_inner} and show
\begin{align*}
dC(K) = \inner{\sum_{t=0}^\infty \mathcal{L}^t(dK^{\cT}\mult L^K + (L^K)^{\cT} \mult dK)}{X(0)}
\end{align*}
Since $\mathcal{T}$ is the adjoint operator of $\mathcal{L}$, we have
\begin{align*}
dC(K)     &= \inner{(dK^{\cT}\mult L^K + (L^K)^{\cT} \mult dK)}{\sum_{t=0}^\infty \mathcal{T}^t(X(0))} \\
    &= \inner{(dK^{\cT}\mult L^K + (L^K)^{\cT} \mult dK)}{\X^K}.
\end{align*}
Notice that each block in $\X^K$ is symmetric. We can easily verify the following fact
\begin{align*}
    \inner{dK^{\cT}\mult L^K}{\X^K}=\inner{(L^K)^{\cT} \mult dK}{\X^K}=\inner{L^K \mult \X^K}{dK}.
\end{align*}
Therefore, we have
$dC(K)    =2\inner{L^K\mult \X^K}{dK}$.
This immediately leads to the desired conclusion due to the fact $dC(K) = \inner{\nabla C({K}) }{dK}$.
\end{proof}

We will also need an explicit formula for the Hessian of the cost.
To avoid tensors, we restrict analysis with the quadratic form of the Hessian $\nabla^2C(K)[E, E]$ on a matrix sequence $E \in \M_{k\times d}^{N_s}$.
\begin{lemma}
For $K \in \K$, the Hessian of the MJLS LQR cost $C({K})$ applied to a direction $E \in \M_{k \times d}^{N_s}$ is given by
\begin{align}
\label{eq:hessian}
\nabla^2 C(K) [E, E] &= 2\inner{(R+B^{\cT} \mult \E{P^K}\mult  B)\mult E\mult \X^K}{E} \nonumber \\
&\quad\  - 4\inner{B^{\cT}\mult  \E{(P^K)'[E]}{\mult} \varGamma\mult \X^K}{E}
\end{align}
where
\begin{align}
(P^K)'[E] = \sumtinf \mathcal{L}^t \big( E^{\cT} \mult L^K + (L^K)^{\cT} \mult E \big)
\end{align}
\end{lemma}
\begin{proof}
Recall that $\varGamma_i \coloneqq A_i - B_i K_i$ with $\varGamma = (\varGamma_1, \ldots, \varGamma_{N_s})$.
Applying the Taylor series expansion about $E$ \cite{dattorro2010convex}, we have that the quadratic form of the hessian $\nabla^2 C(K)$ on $E$ is given by
\begin{equation}
\nabla^2 C(K) [E, E] = \left.\frac{d^2}{dt^2}\right|_{t=0} C(K + tE).
\end{equation}
Using~\eqref{eq:cost_inner}, we then have
\begin{equation}
\nabla^2 C(K) [E, E]  = \inner{\left.\frac{d^2}{dt^2}\right|_{t=0} P^{K+tE}}{X(0)}
\end{equation}
Denote $(P^K)'[E] \coloneqq \left.\frac{d}{dt}\right|_{t=0} P^{K+tE}$. From Lemma~\ref{lemma:policy_grad} we have
\begin{align*}
(P^K)'[E] = \sumtinf \mathcal{L}^t \Big( &E^{\cT} \mult (R\mult K-B^{\cT} \mult \E{P^K} \mult \varGamma) \nonumber \\
    & + (R\mult K-B^{\cT} \mult \E{P^K} \mult \varGamma)^{\cT} \mult E \Big)
\end{align*}
We can show that
\begin{equation}
\left.\frac{d^2}{dt^2}\right|_{t=0} P^{K+tE} = \sumtinf \mathcal{L}^t \big( S \big)  
\end{equation}
where
\begin{align} \label{eq:S}
S &= - 2\,\big(E^{\cT} \mult B^{\cT} \mult \E{(P^K)'[E]} \mult \varGamma + \varGamma \mult \E{(P^K)'[E]} \mult B \mult E  \big) \nonumber \\
&\qquad + 2\,E^{\cT}\mult (R + B^{\cT} \mult \E{P^K} \mult B) \mult E
\end{align}

Since $\mathcal{T}$ is the adjoint operator of $\mathcal{L}$, we have
\begin{align*}
\nabla^2 C(K) [E, E]  &= \inner{\sumtinf \mathcal{L}^t \big( S \big)}{X(0)} \\
    &= \inner{S}{\sumtinf \mathcal{L}^t \big( X(0) \big)}  = \inner{S}{\X^K}.
\end{align*}
Plugging~\eqref{eq:S} into the above, we get
\begin{align*}
\nabla^2 C(K) [E, E] &= 2\inner{E^{\cT}\mult (R + B^{\cT} \mult \E{P^K} \mult B) \mult E}{\X^K} \\
&\qquad\quad - 2 \inner{E^{\cT} \mult B^{\cT} \mult \E{(P^K)'[E]} \mult \varGamma}{\X^K} \\
&\qquad\quad - 2 \inner{\varGamma \mult \E{(P^K)'[E]} \mult B \mult E}{\X^K} \\
&= 2\inner{E^{\cT}\mult (R + B^{\cT} \mult \E{P^K} \mult B) \mult E}{\X^K} \\
&\qquad\quad -4\inner{E^{\cT} \mult B^{\cT} \mult \E{(P^K)'[E]} \mult \varGamma}{\X^K}
\end{align*}
We can get the desired result by noting that each block in $\X^K$ is symmetric and using the cyclic property of the trace.
\end{proof}

\textbf{Optimization Landscape for MJLS.}
Now we are ready to discuss the optimization landscape for the MJLS LQR problem.
Notice that LTI systems are just a special case of MJLS.
Since  policy optimization for quadratic control of LTI systems is non-convex, the same is true for the MJLS case.
However, from our gradient formula in Lemma \ref{lemma:policy_grad}, we can see that as long as \(\Expx{x_0 x_0^T}\) is 
full rank and $\pi_i>0$ for all $i$, it is necessary that a stationary point given by $\nabla C({K})=0$   satisfies
\begin{equation*}
L_i^K = \left(R_i + B^T_i \Ei{P^K} B_i \right) K_i - B^T_i \Ei{P^K} A_i=0.
\end{equation*}
Substituting the above equation into the coupled Lyapunov equation \eqref{eq:lyap_markov} leads to the global solution ${K}^*$ defined by the coupled Algebraic Riccati Equations~(\ref{eq:markov_riccati}).
Therefore, the only stationary point is the global optimal solution. Overall, the optimization landscape for the MJLS case is quite similar to the classic LQR case if we allow the initial mode to be sufficiently random, i.e. $\pi_i>0$ for all $i$. Based on such similarity, it is reasonable to expect that the local search procedures (e.g. policy gradient) will be able to find the unique global minimum ${K}^*$  for MJLS despite the non-convex nature of the problem. Compared with the LTI case, the characterization of $\K$ is more complicated for MJLS. Hence the main technical issue is how to show gradient-based methods can handle the feasibility constraint ${K}\in \K$ without using projection.

\textbf{Key Properties of the MJLS LQR Cost.} To analyze the performance of gradient-based methods for the MJLS LQR problem, a few key properties of $C(K)$ will be used. 
By assumption, we have $\mu \coloneqq \min_{i\in \Omega}(\pi_i)\, \sigma_{\min}\!\left( \Expx{x_0 x_0^T} \right)>0$. 
Then, we can identify several key properties of $C(K)$ as follows.

\begin{lemma}\label{lemma:coercive}
The cost~\eqref{eq:markov_cost} satisfies the following properties:
\begin{enumerate}
    \item Coercivity: The cost function $C$ is coercive in the sense that for any sequence $\{K^l\}_{l=1}^\infty\subset \K$ we have
\begin{equation*}
    C(K^l) \rightarrow +\infty 
\end{equation*}
if either $\|K^l\|_2 \rightarrow +\infty$, or  $K^l$ converges to an element $K$ in the boundary $\partial \K$.
\item Almost smoothness: Given elements $K, \, K' \in \K$,
the cost function \(C(K)\) defined in~\eqref{eq:switched_cost} satisfies
\begin{align*}
&C(K') - C(K) \\
&\qquad = \sum_{i\in\Omega} \left( -2\tr{\X_i^{K'} \Delta K_i^T L_i^K} \right.\\
&\quad\qquad \left.+ \tr{\X_i^{K'} \Delta K_i^T \left(R_i + B_i^T \Ei{P^K} B_i\right) \Delta K_i} \right)\\
&\qquad = -2\inner{\Delta K^T \mult L^K}{\X^{K'} } + \inner{\Delta K^T\Psi \mult \Delta K}{\mult\X^{K'} },
\end{align*}
where $\Delta K_i = (K_i - K_i')$ and $\Psi \coloneqq R + B^{\cT}\mult \E{P^K}\mult B $.
\item Gradient dominance: Given the optimal policy $K^*$, the following sequence of inequalities holds for any $K\in \K$:
\begin{align*}
C(K) - C(K^*) &\leq  \| \X^{K^*} \|_{\max} \inner{\Psi^{\cinv} \mult L^K}{L^K}\\
&\leq \frac{\| \X^{K^*} \|_{\max}}{\Lambda_{\min}(R)} \|L^K\|_2^2 \\
&\leq \frac{\| \X^{K^*} \|_{\max}}{4\mu^2 \Lambda_{\min}(R)} \|\nabla C(K)\|_2^2.
\end{align*}
\item Compactness of the sublevel sets: The sublevel set defined as $\K_{\alpha} \coloneqq \{ K\in\K : C(K) \leq \alpha \}$ is compact for every $\alpha \ge C(K^*)$.
\item Smoothness on the sublevel sets: For any sublevel set $\K_\alpha$, choose the smoothness constant as 
\begin{align*}
\hspace{-0.1in}L=2\left(\! \|R\|_{\max} + \|B\|^2_{\max} \left(1 + \frac{2\xi}{\|B\|_{\max}}\right)\! \frac{\alpha}{\mu} \right)\! \frac{\alpha}{\Lambda_{\min}(Q)},
\end{align*}
where $\xi$ is calculated as
\begin{equation*}
\xi = \frac{1}{\Lambda_{\min}(Q)} \left( \frac{1 + \|B\|_{\max}^2}{\mu} \alpha + \|R\|_{\max} \right)-1.
\end{equation*}
Then for any $K\in \K_\alpha$, we have
$\|\nabla^2 C(K)\|\le L$.
In addition, for any $(K, K')$ satisfying $tK+(1-t)K'\in \K_\alpha$ $\forall t\in [0,1]$, the following inequality holds
\begin{align}
\label{eq:smooth1}
    C(K')\le C(K)+ \inner{\nabla C(K)}{K'-K}+\frac{L}{2}\|K'-K\|_2^2.
\end{align}
\end{enumerate}
\end{lemma}
\begin{proof}
To prove Statement 1,  first notice that we have
\begin{align*}
C(K^l) &\geq \Expx{\sumomega \pi_i x_0^T (Q_i + (K_i^l)^T R_i K_i^l) x_0} \\
&\geq \mu \Lambda_{\min}(R) \|K^l\|_2^2.
\end{align*}  
This directly shows that $C(K^l) \rightarrow +\infty$ as $\|K^l\|_2\rightarrow +\infty$.
Next, we assume $K^l\rightarrow K\in \partial \K$. 
Based on Proposition~\ref{prop:mss}, we know that for all $l$, there exists $Y^l \succ 0$ such that
\begin{equation*}
Y^l - \mathcal{L}^{K^l}(Y^l) = Q + (K^l)^{\cT} \mult R \mult K^l
\end{equation*}
where the dependence of $\mathcal{L}$ on $K$ is emphasized by the superscript. 
We now want to show that the sequence $\{Y^l\}$ is unbounded, and will use a contradiction argument.
Suppose that $\{Y^l\}$ is bounded. By the Weierstrass-Bolzano theorem, $\{Y^l\}$ admits a subsequence $\{Y^{l_n}\}_{n=0}^\infty$ which converges to some limit point denoted as $Y$. Clearly, we have $Y\succeq 0$. 
For the same subsequence $\{l_n\}_{n=0}^\infty$, we have $\lim_{n\rightarrow \infty} K^{l_n}=K\in \partial \K$.
For all $l_n$, we still have
\begin{equation*}
Y^{l_n} - \mathcal{L}^{K^{l_n}}(Y^{l_n}) = Q + (K^{l_n})^{\cT} \mult R \mult K^{l_n}
\end{equation*}
Now  letting $n$ go to $\infty$, by continuity, leads to the equation
 $Y - \mathcal{L}^{K}(Y) = Q + K^{\cT} \mult R \mult K$. Since  $Q \succ 0$, $R\succ 0$, and $\mathcal{L}^{K}(Y) \succeq0 $, we conclude $Y\succ 0$.
By Proposition~\ref{prop:mss}, we have $K\in \K$, and this contradicts the fact that $K\in \partial \K$.
Therefore, $\{Y^l\}$ must be unbounded. Since $Y^l$ is positive definite, we can further conclude that $\{\tr{Y^l}\}$ is unbounded and $C(K^l) \rightarrow \infty$. 
This completes the proof of Statement~1.

Next, we prove Statement 2. Recall that we have \( \varGamma_i = A_i - B_i K_i \). To simplify the notation, we denote \( \varGamma_i' \coloneqq A_i - B_i K_i' \). 
By definition, we have
\begin{align}\label{eq:Cdiff}
C(K') - C(K)&= \inner{P^{K'}}{X(0)} - \inner{P^K}{X(0)} \nonumber \\
	&=\inner{P^{K'} - P^K}{X(0)}
\end{align}
Now we develop a formula for $(P_i^{{K}'}\! - P_i^{{K}})$.
Based on~\eqref{eq:lyap_markov}, we have $P_i^{K'} = (\varGamma_i')^T \Ei{P^{K'}} \varGamma_i' + Q_i + (K_i')^T R_i K_i'$.  Using this, we can directly show
\begin{align*}
P_i^{K'} - P_i^{K} &= (\varGamma_i')^T \Ei{P^{K'}} \varGamma_i' + Q_i + (K_i')^T R_i K_i' - P_i^{K} \\
 &= (\varGamma_i')^T \left( \Ei{P^{K'}} - \Ei{P^K} \right) (\varGamma_i')+ Q_i \\
&  \quad\; + (\varGamma_i')^T \Ei{P^K} \varGamma' + (K')^T R K' - P_i^{K}\\
 &= (\varGamma_i')^T \left( \Ei{P^{K'} - P_i^K} \right) \varGamma_i' \\
&  \quad\; + (K_i - K_i') (R_i + B_i^T \Ei{P^K} B_i)(K_i - K_i')\\
&  \quad\; - (K_i - K_i')^T\left(R_i K_i - B_i^T \Ei{P^K} \varGamma_i\right) \\
&  \quad\; - \left(R_i K_i - B_i^T \Ei{P^K} \varGamma_i\right)^T(K_i - K_i') 
\end{align*}
Since $K'\in \K$, the above coupled Lyapunov equations can be directly solved as
\begin{align*}
P^{K'} - P^{K} &=\! \sum_{t = 0}^\infty (\mathcal{L}^{K'})^t \left( \Delta K^{\cT} \mult (R + B^{\cT} \mult \E{P^K}\mult B) \mult\Delta K \right. \\
&\qquad \qquad \qquad \left.+  \Delta K^{\cT}\mult L^K + (L^K)^{\cT} \mult \Delta K \right)
\end{align*}
Here the notation $\mathcal{L}^{K'}$ emphasizes that this is the operator associated with $K'$.
Now we can prove Statement 2 by substituting the above formula into \eqref{eq:Cdiff} and applying the fact that $\mathcal{T}^{K'}$ is the adjoint operator of $\mathcal{L}^{K'}$.

To prove Statement 3, we will make use of the almost smoothness condition. We can rewrite the almost smoothness condition and complete the squares as follows
\begin{align*}
&C(K')-C(K)\\
&\qquad=\inner{- \Delta K^T L^K -(L^K)^T \Delta K+ \Delta K^T \mult \Psi \mult \Delta K}{ \X^{K'} } \\
&\qquad= \inner{\left(-\Delta K +\Psi^{\cinv} \mult L^K\right)^T\Psi \mult \left(-\Delta K \!+\Psi^{\cinv} \mult L^K \right)}{ \mult \X^{K'}} \\
&\qquad\qquad - \inner{(L^K)^T\Psi^{\cinv} \mult L^K }{ \mult \X^{K'}}
\end{align*}
We know $\left(-\Delta K +\Psi^{\cinv} \mult L^K\right)^T\Psi \mult \left(-\Delta K \!+\Psi^{\cinv} \mult L^K \right)\succeq 0$ and $\X^{K'}\succeq 0$. Hence we have 
\begin{align*}
    \inner{\left(-\Delta K +\Psi^{\cinv} \mult L^K\right)^T\Psi \mult \left(-\Delta K \!+\Psi^{\cinv} \mult L^K \right)}{ \mult \X^{K'}}\ge 0,
\end{align*}
which leads to the following inequality for any $K'\in\K$:
\begin{align*}
C(K')-C(K)\geq - \inner{(L^K)^T\Psi^{\cinv} \mult L^K }{ \mult \X^{K'}}
\end{align*}
Then we can set $K'=K^*$ to show
\begin{align*}
C(K) - C(K^*) &\leq \inner{(L^K)^T\mult\Psi^{\cinv} \mult L^K}{\X^{K^*}}\\
&\leq \| \X^{K^*} \|_{\max} \inner{\Psi^{\cinv} \mult L^K}{L^K} \\
&\leq \frac{\| \X^{K^*} \|_{\max}}{\Lambda_{\min}(R)} \|L^K\|_2^2 \\
&= \frac{\| \X^{K^*} \|_{\max}}{4\Lambda_{\min}(R)} \|\nabla C(K) \mult (X^K)^{\cinv}\|_2^2\\
&\leq \frac{\| \X^{K^*} \|_{\max}}{4\Lambda_{\min}(\X^K)^2\Lambda_{\min}(R)} \|\nabla C(K)\|_2^2 \\
&\leq \frac{\| \X^{K^*} \|_{\max}}{4\mu^2 \Lambda_{\min}(R)} \|\nabla C(K)\|_2^2
\end{align*}
This leads to the desired conclusion in Statement 3.

Statement 4 can be proved using  the continuity and coercivity of $C(K)$.
With the coercive property in place, we can continuously extend the function domain from $\K$ to $\M_{k\times d}^{N_s}$ by allowing $\infty$ as a function value.
Based on Proposition 11.12 in \cite{bauschke2011convex}, we know that
$\K_\alpha$ is bounded for any finite $\alpha$. Since $C(K)$ is continuous on $\K$, the set $\K_\alpha$ is also closed. 
Hence Statement 4 holds as desired.

Finally, to prove Statement 5, we only need to bound the norm of $\nabla^2 C(K)$. Then the desired conclusion follows by applying the mean value theorem. Since $\nabla^2 C(K)$ is self-adjoint, its operator norm can be characterized as
\begin{align*}
\|\nabla^2 C(K)\| = \sup_{\|E\|_2 = 1} |\nabla^2 C(K)[E,E] |
\end{align*}
Based on the Hessian formula \eqref{eq:hessian}, we have
\begin{align}
\label{eq:keybound1}
\begin{split}
&\sup_{\|E\|_2 = 1} |\nabla^2 C(K)[E,E] | \\
&\;\qquad\qquad\qquad \leq 2 \sup_{\|E\|_2 = 1} |\inner{(R+B^{\cT}\mult \E{P^K}\mult B)\mult E \mult \X^K}{E}| \\
&\quad\qquad\qquad\qquad\qquad  + 4\sup_{\|E\|_2 = 1} |\inner{B^{\cT} \mult \E{(P^K)'}\mult \varGamma \mult \X^K}{E}|
\end{split}
\end{align}
Now we only need to provide upper bounds for the two terms on the right side of the above inequality. For simplicity,
we denote $q_1\coloneqq\sup_{\|E\|_2 = 1} |\inner{(R+B^{\cT}\mult \E{P^K}\mult B)\mult E \mult \X^K}{E}|$ and $q_2\coloneqq\sup_{\|E\|_2 = 1} |\inner{B^{\cT} \mult \E{(P^K)'}\mult \varGamma \mult \X^K}{E}|$. As a matter of fact, $q_1$ and $q_2$ can be bounded as follows
\begin{align}
\label{eq:q1}
    q_1&\le \left(\|R\|_{\max} + \|B\|^2_{\max}  \frac{C(K)}{\mu} \right) \frac{C(K)}{\Lambda_{\min} (Q)}\\
    \label{eq:q2}
    q_2&\le\frac{\xi \|B\|_{\max}C(K)^2}{\mu\Lambda_{\min}(Q)}
\end{align}
The proofs of~\eqref{eq:q1} and~\eqref{eq:q2} are tedious and hence are deferred to the appendix for readability.
Now we are ready to prove the $L$-smoothness of $C(K)$ within the set $\K_\alpha$.
Notice $C(K)\le \alpha$ for any $K\in \K_\alpha$. Hence we can combine \eqref{eq:q1} and \eqref{eq:q2}  to show $2q_1+4q_2\le L$ where $L$ is given in Statement 5. Based on the mean value theorem, this leads to the desired conclusion. It is worth emphasizing that~\eqref{eq:smooth1} only holds when the line segment between $K$ and $K'$ is in $\K_\alpha$. Since $\K_\alpha$ is non-convex in general, it is possible that there exists $K,K'\in \K_\alpha$ such that~\eqref{eq:smooth1} does not hold. 
\end{proof}

Now we briefly explain the importance of the above properties.
When applying an iterative optimization method to search for $K^*$, two issues need to be addressed and our techniques will heavily rely on the above cost properties.
\begin{enumerate}
    \item Feasibility: One has to ensure that the iterates generated by the optimization method always stay in the non-convex feasible set $\K$. The coercivity implies that the function $C(K)$ serves as a barrier function on $\K$. Based on the coercivity and the compactness of the sublevel set, one can show that the decrease of the cost ensures the next iterate to stay inside $\K$.
    \item Convergence: After ensuring the feasibility, one next needs to show that the iterates generated by the optimization method converge to $K^*$.  The smoothness and gradient dominance properties will play a key role in the convergence proof when there is an absence of convexity.
\end{enumerate}

 In the next section, we will present three gradient-based optimization methods for the MJLS LQR problem and provide global convergence guarantees.

\section{Algorithms and Convergence}
\label{sec:discrete_flow}

In Section \ref{sec:LQRreview}, we have reviewed three optimization methods (namely the policy gradient method, the Gauss-Newton method, and the natural gradient method) for the LQR problem. In this section, we will consider these algorithms in the MJLS LQR setting and provide new global convergence guarantees. For the readers who are interested in the (continuous-time) ODE limits of these methods, we present some relevant results in the appendix.

\subsection{Policy Gradient Method}
The gradient descent method is arguably the simplest optimization algorithm.
In the MJLS LQR setting, 
the gradient descent method iterates as
\begin{equation}\label{eq:gd}
K^{n+1}= K^n-\eta \nabla C(K^n),
\end{equation}
where $K^0$ is required to be in $\K$.
As mentioned before, we first need to ensure that the iterates generated by \eqref{eq:gd} are always in $\K$.
Consider the one-step policy gradient update $K'\leftarrow K-2\eta  L^K \mult \X^K$.
We will use the coercivity of $C(K)$ and the compactness of $\K_\alpha$ to show that  for every $K\in \K$, we can choose some fixed  $\eta$ such that
$K'$ will also be in $\K$.

\begin{lemma}\label{lemma:gd_stability}
Suppose $K\in \K_\alpha$ and $K'=K-\eta \nabla C(K)$. Set $L$ as described in Statement 5 of Lemma~\ref{lemma:coercive}. If $\eta\le \frac{1}{L}$, then we have $K'\in \K_\alpha\subset \K$ and
\begin{align}
\label{eq:key1}
    C(K')\le  C(K)-\frac{\eta}{2}\|\nabla C(K)\|_2^2.
\end{align}
\end{lemma}
\begin{proof}

We define the interior set  of $\K_\alpha$ as $\K_\alpha^{\mathrm{o}}\coloneqq\{K\in \K:C(K)<\alpha\}$. The complement of $\K_\alpha^{\mathrm{o}}$ is denoted as $(\K_\alpha^{\mathrm{o}})^c$.
Notice $\|\nabla^2 C(K)\|\le L$ for all $K\in \K_\alpha$. By continuity, 
there exists $\epsilon>0$ such that $\|\nabla^2 C(K)\|\le 1.1L$ for all $K\in \K_{\alpha+\epsilon}$.

Clearly $(\K_{\alpha+\epsilon}^{\mathrm{o}})^c$ is a closed set and $(\K_{\alpha+\epsilon}^{\mathrm{o}})^c\cap \K_\alpha=\emptyset$.
Since $\K_\alpha$ is compact, we know the distance between $\K_\alpha$ and $(\K_{\alpha+\epsilon}^{\mathrm{o}})^c$ is strictly positive. We denote this distance as $\delta$. Let us choose $\tau=\min\{0.9\delta/\norm{\nabla C(K)} \,,1/11L\}$. Obviously, the line segment between $K$ and $(K-\tau\nabla C(K))$ is in $\K_{\alpha+\epsilon}$. Notice $\|\nabla^2 C(K)\|\le 1.1L$ for all $K\in \K_{\alpha+\epsilon}$, and hence we have
\begin{align*}
C(K\!-\tau\nabla C(K)) &\le  C(K)+ \inner{\nabla C(K)}{K\!-\tau\nabla C(K)-K}\\
&\qquad+\frac{1.1L}{2}\|K-\tau\nabla C(K)-K\|_2^2
\end{align*}
which leads to 
\begin{align*}
    C(K-\tau\nabla C(K))\le  C(K)+ \left(-\tau+\frac{1.1L\tau^2}{2}\right)\|\nabla C(K)\|_2^2
\end{align*}
As long as  $\tau\le 2/(1.1L)$, we have $-\tau+\frac{1.1L\tau^2}{2}\le 0$ and $C(K-\tau\nabla C(K))\le C(K)$. Hence we have $K-\tau\nabla C(K)\in \K_\alpha$. Actually, it is straightforward to see that the line segment between $K$ and $(K-\tau \nabla C(K))$ is in $\K_\alpha$.

The rest of the proof follows from induction.
We can apply the same argument to show that the line segment between $(K-\tau\nabla C(K))$ and $(K-2\tau\nabla C(K))$ is also in $\K_\alpha$. This means that the line segment between $K$ and $(K-2\tau\nabla C(K))$ is in $\K_\alpha$. Since $\tau>0$, we only need to apply the above argument for finite times and then will be able to show that the line segment between $K$ and $(K-\eta \nabla C(K))$ is in $\K_\alpha$ for any $\eta\le \frac{1}{L}$.\footnote{The argument even works for any $\eta\le \frac{2}{1.1L}$. Since the step size leading to the fastest convergence rate is $\frac{1}{L}$, we state our result only for $\eta\le \frac{1}{L}$.} Since $\|\nabla^2 C(K)\|\le L$ for all $K\in \K_\alpha$, we have
\begin{align*}
    C(K')&\le  C(K)+ \inner{\nabla C(K)}{K'-K}+\frac{L}{2}\|K'-K\|_2^2\\
    &= C(K)+\left(-\eta+\frac{L\eta^2}{2}\right)\|\nabla C(K)\|_2^2\\
    &\le C(K)-\frac{\eta}{2}\|\nabla C(K)\|_2^2,
\end{align*}
where the last step follows from the fact that we have $\eta\le \frac{1}{L}$.
This completes the proof.
\end{proof}

Next, we can combine \eqref{eq:key1} with the gradient dominance property to show that the cost associated with the one-step progress of the gradient descent method is decreasing. This step is quite standard.

\begin{lemma}\label{lemma:one_step_gd}
Suppose $K\in \K_\alpha$ and $K' = K - \eta \nabla C(K)$. Set $L$ as described in Statement 5 of Lemma~\ref{lemma:coercive}.
If \( \eta \leq \frac{1}{L}\),  then the following inequality holds
\begin{align*}
C(K') - C(K^*) \leq \left( 1 - \frac{2\mu^2 \Lambda_{\min}(R)}{\| \X^{K^*} \|_{\max}} \eta \right)\! \left( C(K) - C(K^*) \right)\!.
\end{align*}
\end{lemma}
\begin{proof}
By  Lemma~\ref{lemma:gd_stability}, we know $K'$ is stabilizing. We can combine~\eqref{eq:key1} with Statement 3 in Lemma~\ref{lemma:coercive} to show
\begin{align*}
    C(K')- C(K)&\le -\frac{\eta}{2}\|\nabla C(K)\|_2^2\\
    &\le -\frac{2\mu^2 \Lambda_{\min}(R)\eta}{\| \X^{K^*} \|_{\max}}\left(C(K)-C(K^*)\right)
\end{align*}
which directly leads to the desired conclusion.
\end{proof}

Now we are ready to prove the global convergence of the policy gradient method \eqref{eq:gd}.

\begin{thm} \label{thm:gd_conv}
Suppose \(K^0\in \K\). Choose $\alpha=C(K^0)$ and set $L$ as described in Statement 5 of Lemma~\ref{lemma:coercive}.
For any step size \( \eta \leq \frac{1}{L}\), the iterations generated by the gradient descent method~\eqref{eq:gd} always stay in $\K$ and converge to the global minimum $K^*$ linearly as follows
\begin{align}
\label{eq:mainConGD}
C(K^n) - C(K^*) &\leq \left( 1 - \frac{2\mu^2 \Lambda_{\min}(R)}{\| \X^{K^*} \|_{\max}} \eta \right)^n \times \nonumber \\
& \qquad \qquad \left(C(K^0) - C(K^*) \right).
\end{align}
\end{thm}
\begin{proof}
We will use an induction argument. Since $\alpha=C(K^0)$, we have $K^0\in \K_\alpha$. By Lemma~\ref{lemma:one_step_gd}, we know~\eqref{eq:mainConGD} holds for $n=1$. Since $C(K^1)\le C(K^0)$,  we have $K^1\in \K_\alpha$. We can apply Lemma \ref{lemma:one_step_gd} again to show \eqref{eq:mainConGD} holds for $n=2$. Now it is obvious that we can repeatedly apply the above argument to show~\eqref{eq:mainConGD} holds for any $n$.
\end{proof}

From the above proof, one can see that without using projection, one can still guarantee the gradient descent method will stay in the feasible set and converge to the global minimum. When the model is unknown, one can directly apply policy-based learning techniques such as zeroth-order optimization methods~\cite{conn2009introduction,nesterov2017random} to estimate the gradient $\nabla C(K)$ from data.
Our result implies that such data-driven methods will work if the gradient is estimated with some reasonable accuracy. A detailed quantification of the sample complexity of such learning methods in the MJLS setting is beyond the scope of our paper and will be investigated in the future.

\subsection{Gauss-Newton Method}
The step size choice of the policy gradient method depends on the parameter $L$ which is typically unknown.
Next, we will consider the Gauss-Newton method whose step size selection is much more straightforward. The Gauss-Newton method iterates as follows
\begin{equation}
\label{eq:gngd}
K^{n+1}= K^n-2\eta (R + B^{\cT}\mult \E{P^{K^n}} \mult B)^{\cinv} \mult L^{K^n},
\end{equation}
where the initial policy $K^0$ is required to be in the set $\K$.

Again, we need to ensure that the iterates generated by \eqref{eq:gngd} are always in the set $\K$.
Consider the one-step Gauss-Newton update $K'\leftarrow K-2\eta (R + B^{\cT}\mult \E{P^{K}} \mult B)^{\cinv} \mult L^K$.
 We need to show
that for every $K\in \K$, we can choose a step size $\eta$ such that
$K'$ will also be stabilizing the closed-loop
dynamics in the mean-square sense. This is formalized as below.

\begin{lemma}\label{lemma:step_stable_GN} 
Suppose \(K\in \K\).
Then the one-step update $K'$ obtained from the Gauss-Newton method~\eqref{eq:gngd} will also be in $\K$ if the step size \(\eta\) satisfies $\eta \leq \frac{1}{2}$.
\end{lemma}
\begin{proof}
One may modify the proof for Lemma~\ref{lemma:gd_stability} to prove the above result. Here we present an alternative proof based on
 Lyapunov theory. Such a proof will also highlight a key difference between the policy gradient method and the Gauss-Newton method.
The main idea here is that for Gauss-Newton method, the value function at the current
step serves naturally as a Lyapunov function for the next
move due to the positive definiteness of $Q_i$.

Recall from Proposition~\ref{prop:mss} that the controller $K'$ stabilizes~\eqref{eq:cl} in the mean-square sense if and only if there exists \( Y\succ 0 \in \M_{d\times d}^{N_s}\) such that
\begin{equation}
\label{eq:LMI}
\mathcal{L}^{K'}(Y) - Y \prec 0
\end{equation}
We will show that the above condition can be satisfied by setting $Y=P^K$ where $P^K$ solves the coupled MJLS Lyapunov equations $(A_i - B_i K_i)^T \Ei{P^K} (A_i - B_i K_i) + Q_i + K_i^T R_i K_i = P_i^K$ for every $i\in\Omega$. 
Notice that the existence of $P^K$ is guaranteed by the assumption $K\in \K$.
Denote $\Delta K_i \coloneqq K_i-K_i'$. 
The Lyapunov equation for $P^K$ can be rewritten as $(A_i - B_i K_i' - B_i \deltaKi )^T \Ei{P^K} (A_i - B_i K_i' - B_i \deltaKi)+ Q_i + (K_i' + \deltaKi)^T R_i (K_i' + \deltaKi) = P_i^K $. 
From this, we can directly obtain 
\begin{align*}
& (A_i - B_i K_i')^T \Ei{P^K} (A_i - B_i K_i') - P_i^K = \\
&\qquad\qquad\qquad -\left( Q_i + (K_i')^{T}  R_i  K_i'\right) \\
&\qquad\qquad\qquad - \left(\deltaKi^{T} R_i \deltaKi + \deltaKi^{T} B_i^T \Ei{P^K} B_i \deltaKi \right) \\
&\qquad\qquad\qquad - \deltaKi^{T} \left( R_i K_i' - B_i^{T} \Ei{P^K} (A_i - B_i K_i') \right) \\
&\qquad\qquad\qquad - \left( R_i K_i' - B_i^{T} \Ei{P^K} (A_i - B_i K_i') \right)^{T} \deltaKi
\end{align*}
Since \((R, P^K, Q)\) are all positive definite, the sum of the first two terms on the right hand side is negative definite.
We only need the last two terms to be negative semidefinite.
Note that, for the Gauss-Newton method, $\deltaKi = 2 \eta (R_i + B_i^T \Ei{P^K} B_i)^{-1} L_i^K$. We have
\begin{align*}
&\deltaKi^T \left( R_i K_i' - B_i^T \Ei{P^K} (A_i - B_i K_i') \right) \\
&\qquad = \deltaKi^T \left( \left(R_i +  B_i^T \Ei{P^K} B_i\right) K_i' - B_i^T \Ei{P^K} A_i \right) \\
&\qquad= \deltaKi^T \left( - \left(R_i +  B_i^T \Ei{P^K} B_i\right) \deltaKi + L_i^K \right) \\
&\qquad= 2 \eta(1 - 2\eta) (L_i^K)^T\left(R_i +  B_i^T \Ei{P^K} B_i\right)^{-1}  L_i^K
\end{align*}
which is positive semidefinite under the condition $\eta\le \frac{1}{2}$ for all $i\in \Omega$.
\end{proof}

From the above proof, we can clearly see that $P^K$ can be used to construct a Lyapunov function for $K'$
if $\eta\le \frac{1}{2}$. This leads to a novel proof for the stability along
the Gauss-Newton iteration path. This idea may even be extended for problems where the cost is not coercive. For example, a similar idea has been used to show the convergence
properties of policy optimization methods for the mixed
$\mathcal{H}_2/\mathcal{H}_\infty$ state feedback design problem where the cost function may not
blow up to infinity on the boundary of the feasible set \cite{zhang2019policyc}.
It is worth mentioning that the above proof idea does not work for the gradient descent method. For the policy gradient method, the value function at step $n$ cannot be directly used as a Lyapunov function at step $(n+1)$. A similar fact has also been observed for the mixed
$\mathcal{H}_2/\mathcal{H}_\infty$ control problem~\cite{zhang2019policyc}.

Next, we will apply the ``almost smoothness" condition and the gradient dominance condition stated in Lemma \ref{lemma:coercive} to show that cost associated by the one-step progress of the Gauss-Newton method is non-increasing.

\begin{lemma}\label{lemma:one_stepGN}
If $K'= K-2\eta \Psi^{-1}L^K$ with  $\Psi_i \coloneqq (R_i + B_i^T \Ei{P^K} B_i)$ and \(\eta \leq \frac{1}{2}\), then the following inequality holds
\begin{align*}
C(K') - C(K^*) &\leq \left( 1 -  \frac{2\mu}{\| \X^{K^*} \|_{\max}} \eta \right)\left( C(K) - C(K^*) \right).
\end{align*}
\end{lemma}
\begin{proof}
Based on Lemma~\ref{lemma:step_stable_GN}, we know $K'\in \K$.
Based on Statement 2 (the almost smoothness condition) in Lemma \ref{lemma:coercive},  we have
\begin{align*}
C(K') - C(K) &= -4\eta\inner{ (L^K)^T\Psi^{\cinv} \mult L^K}{ \X^{K'}}\\
&\qquad \qquad \quad +4\eta^2 \inner{(L^K)^T\Psi^{\cinv}\mult L^K}{\X^{K'}}\\
	&\leq -2\eta \inner{(L^K)^T \Psi^{\cinv} \mult  L^K}{ \X^{K'}} \\
	&\leq -2\eta \Lambda_{\min}(\X^{K'}) \inner{\Psi^{\cinv}\mult L^K}{L^K} \\
	&\leq -2\eta \mu \inner{\Psi^{\cinv} \mult L^K}{L^K} \\
	&\leq -\frac{2\eta\mu}{\| \X^{K^*} \|_{\max}} \left(C(K) - C(K^*)\right),
\end{align*}
where the last step follows from Statement 3 in Lemma~\ref{lemma:coercive}.
\end{proof}

We are now ready to present the linear rate bound on the convergence of the Gauss-Newton method.

\begin{thm} \label{thm:gn_conv}
Suppose \(K^0\in \K\).
For any step size $\eta \le \frac{1}{2}$, the iterations generated by
the Gauss-Newton method~\eqref{eq:gngd} always stay in $\K$ and will converge to the global minimum $K^*$ linearly as follows
\begin{align}
\label{eq:mainConGN}
C(K^n) - C(K^*) \le\! \left( 1 -  \frac{2\mu}{\| \X^{K^*} \|_{\max}} \eta\right)^{\!n}\hspace{-0.06in}(C(K^0) - C(K^*) ).
\end{align}
\end{thm}
\begin{proof}
Since Lemma~\ref{lemma:one_stepGN} holds for any $\eta \leq \frac{1}{2}$, we have the following contraction at every step:
\begin{align*}
C({K}^{n+1}) - C({K}^*) \!\leq\! \left(\! 1 -  \frac{2\eta\mu}{\| \X^{K^*} \|_{\max}\!} \right)\! \left( C({K}^n) - C({K}^*) \right)\!.
\end{align*}
Hence, we can obtain the final result using induction.
\end{proof}

From the above result, it is obvious that one can just choose $\eta=\frac{1}{2}$ for the Gauss-Newton method. Actually, such a step size choice leads to the famous policy iteration algorithm in the reinforcement learning literature. One potential drawback of the Gauss-Newton method is that the computation of $\Psi^{-1}L^K$ at every step can be costly.  This motivates the use of the natural policy gradient method.

\subsection{Natural Policy Gradient Method}

The natural policy gradient method avoids the computation of matrix inverse and iterates as follows
\begin{equation}
\label{eq:npgd}
K^{n+1}= K^n-\eta \nabla C(K^n) \mult (\X^{K^n})^{\cinv} = K^n - 2\eta L^{K^n}.
\end{equation}

The proofs for the convergence of the natural policy gradient method are very similar to the Gauss-Newton results presented in the previous section. Consider the one-step natural gradient update $K'\leftarrow K-2\eta L^K$. The following two results state how to choose $\eta$ to ensure $K'\in \K$ and $C(K')\le C(K)$.

\begin{lemma}\label{lemma:step_stable} 
Suppose \(K\in \K\) and $K'\leftarrow K-2\eta L^K$. The one-step natural policy gradient update $K'$  will also in $\K$
if the step size \(\eta\) satisfies 
\begin{equation*}
 \eta \leq \frac{1}{2 \| R + B^{\cT} \mult \E{P^K} \mult B \|_{\max}}.
\end{equation*}
\end{lemma}
\begin{proof}
The proof starts with the same steps as the proof of Lemma~\ref{lemma:step_stable_GN}. 
We will show that the condition~\eqref{eq:LMI} can be met by setting $Y_i=P_i^K$ where $P_i^K$ solves the MJLS Lyapunov equation associated with the controller $K$.
For the natural policy gradient method, we have \(\Delta K \coloneqq 2\eta L^K \).
To show that the last two terms are negative semidefinite, we make the following calculations:
\begin{align*}
&\Delta K_i^T \left( R_i K_i' - B_i^T \Ei{P^K} (A_i - B_i K_i') \right) \\
&= \Delta K_i^T\! \left( (R_i +  B_i^T \Ei{P^K} B_i) K_i' - B_i^T \Ei{P^K} A_i \right) \\
&= \Delta K_i^T\! \left( (R_i +  B_i^T \Ei{P^K} B_i)\! \left(K_i - \Delta K_i \right)\!  -\! B_i^T \Ei{P^K} A_i \right) \\
&= 2\eta (L_i^K)^T \left(L_i^K - 2\eta  \left(R_i +  B_i^T \Ei{P^K} B_i\right) L_i^K \right) \\
&= 2\eta (L_i^K)^T \left(I - 2\eta \left(R_i +  B_i^T \Ei{P^K} B_i\right) \right)  L_i^K
\end{align*}
Clearly, the above term is guaranteed to be positive semidefinite if $\eta$ satisfies
\begin{align*}
 \eta \leq \frac{1}{2\| R_i + B_i^T \Ei{P^K} B_i \|}.
\end{align*}
Lastly, notice
$\| R_i + B_i^T \Ei{P^K} B_i \| \leq \| R + B^{\cT} \mult \E{P^K} \mult B \|_{\max}$ for all $i$.
This leads to the desired conclusion.
\end{proof}

\begin{lemma}\label{lemma:one_step}
Suppose $K\in \K$.
If $K' = K - 2\eta L^{K}$ and   
\begin{equation*}
 \eta \leq \frac{1}{2\| R + B^{\cT} \mult \E{P^K} \mult B \|_{\max}},
\end{equation*}
then the following inequality holds
\begin{align*}
C(K') - C(K^*) &\leq \left( 1 - \frac{ 2\mu \Lambda_{\min}(R)}{\| \X^{K^*} \|_{\max}}\eta \right)\left( C(K) - C(K^*) \right).
\end{align*}
\end{lemma}
\begin{proof}
By Lemma~\ref{lemma:step_stable}, we know $K'\in \K$.
The following bound also holds
\begin{align*}
\inner{(L^K)^T \mult \Psi \mult L^K}{\X^{K'}} \leq  \|\Psi\|_{\max} \inner{L^K}{L^K \mult \X^{K'}} 
\end{align*}
where $\Psi = R + B^{\cT} \mult \E{P^K} \mult B$.
Now we can apply the almost smoothness condition and the gradient dominance condition in Lemma~\ref{lemma:coercive} to show
\begin{align*}
&C(K') - C(K)\\   
&\qquad\qquad\quad = -4\eta \inner{L^K}{L^K \mult \X^{K'}} + 4\eta^2 \inner{\Psi \mult L^K}{L^K \mult \X^{K'}} \\
&\qquad\qquad\quad \le  \left(-4\eta+4\eta^2  \|\Psi\|_{\max}\right) \inner{L^K}{L^K\mult \X^{K'}}\\
&\qquad\qquad\quad \le -2\eta \inner{L^K}{L^K\mult \X^{K'}} \\
&\qquad\qquad\quad \le -2\eta \Lambda_{\min}(\X^{K'}) \|L^K\|^2_2 \\
&\qquad\qquad\quad \le -2\eta \mu \|L^K\|^2_2 \\
&\qquad\qquad\quad \le -\frac{ 2\eta \mu\Lambda_{\min}(R)}{\| \X^{K^*} \|_{\max}} \left(C(K) - C(K^*)\right)
\end{align*}
where the last step follows from Statement 3 in Lemma~\ref{lemma:coercive}. This completes the proof.
\end{proof}

In the above lemmas, the step size depends on $R$, $B$, $P^K$, and $K^n$.
With this in mind, we will fix a constant step size with the help of some bounds in term of the initial cost.
The convergence of the natural policy gradient method is formally given below.

\begin{thm} \label{thm:npg_conv}
Suppose \(K^0\in \K\).
For any step size $\eta \le \frac{1}{2}\left( \| R \|_{\max} + \frac{\| B \|_{\max}^2 C(K^0)}{\mu} \right)^{-1}$,
the iterations generated by the natural policy gradient method~\eqref{eq:npgd} always stay in $\K$ and will converge to the global minimum $K^*$ linearly as follows
\begin{align*}
C(K^n) - C(K^*) \le\! \left(\! 1 - \frac{ 2\mu \Lambda_{\min}(R)}{\| \X^{K^*} \|_{\max}}\eta \right)^{\!n}\!\!\! (C(K^0) - C(K^*)).
\end{align*}
\end{thm}
\begin{proof}
Notice that the following bound holds
\begin{align*}
C(K) 	&\geq \tr{ \sum_{i\in \Omega} \pi_i P_i^{K} } \sigma_{\min}\!\left({\Expx{x_0x_0^T}}\right)\\
	&\geq  \left( \sum_{i\in \Omega} \tr{P_i^{K}} \right)  \min_{i\in\Omega}(\pi_i) \sigma_{\min}\!\left({\Expx{x_0x_0^T}}\right)\\
	& \geq  \mu \| P^K\|_{\max}.
\end{align*}
Therefore, we can apply \eqref{eq:Ebound} to show
\begin{align*}
\|R + B^{\cT} \mult \E{P^K} \mult B\|_{\max} &\leq \|R\|_{\max} + \|B^{\cT} \mult \E{P^K} \mult B\|_{\max}\\
& \leq \|R\|_{\max} + \|B\|^2_{\max} \| P^K\|_{\max} \\
& \leq \|R\|_{\max} + \|B\|^2_{\max} \frac{C(K)}{\mu}.
\end{align*}
which gives an alternative step size bound:
\begin{equation*}
\frac{1}{\| R + B^{\cT} \mult \E{P^K} \mult B \|}_{\max} \geq  \frac{1}{\|R\|_{\max} + \|B\|_{\max}^2\frac{ C(K)}{\mu}}.
\end{equation*}

The rest of the proof can be completed by induction: For the first step we have \(C(K^1) \leq C(K^0)\), which is due to Lemma~\ref{lemma:one_step}.
The proof proceeds by arguing that Lemma~\ref{lemma:one_step} can be applied at every step.
If it were the case that \(C(K^n) \leq C(K^0)\), then
\begin{align*}
\eta &\leq \frac{1}{2\left(\| R \|_{\max} + \frac{\|B \|_{\max}^2 C(K^0)}{\mu}\right)} \\
& \leq  \frac{1}{2\left(\| R \|_{\max} + \frac{\| B \|_{\max}^2 C(K^n)}{\mu}\right)} \\
& \leq \frac{1}{2\| R + B^{\cT} \mult \E{P^{K^n}} \mult B \|_{\max}}
\end{align*}
and so we can apply Lemma~\ref{lemma:one_step} to get
\begin{align*}
C(K^{n+1}) - C(K^*) &\leq \left( 1 - \frac{ 2\eta\mu \Lambda_{\min}(R)}{\| \X^{K^*} \|_{\max}} \right) \times \\
&\qquad \qquad \left( C(K^n) - C(K^*) \right).
\end{align*}
Obviously, now we have \(C(K^{n+1}) \leq C(K^0)\) and can repeat the above argument for the next step.
Therefore, the desired conclusion follows from induction.
\end{proof}

The natural policy gradient method can also be implemented in a model-free manner. When the model is unknown, one can just estimate $\nabla C(K)$ and $\X^K$ from the sampled trajectories. There exist some numerical evidence showing that the natural policy gradient method is significantly faster than the policy gradient method when applied to learn optimal control of large-scale MJLS with unknown parameters~\cite{janschporto2020policy}.

\section{Numerical Experiments}
In this section, we present numerical results to support our proposed theory.
First, we considered a system with 100 states, 20 inputs, and 100 modes.
The system matrices \(A\) and \(B\) were generated using the function \texttt{drss} in MATLAB in order to guarantee that the system would have finite cost with \(K^0 = 0\).
The probability transition matrix \(\Pmat\) was sampled from a Dirichlet Process \( \text{Dir}(99\cdot I_{100} + 1) \).
We also assumed that we had equal probability of starting in any initial mode and that $x_0$ is sampled from a uniform distribution, $\mathcal{U}[-0.5,0.5]$, hence $\mathbb{E}[x_0^T x_0] = \frac{1}{12} I$.
For simplicity we set \( Q_i = I \) and \( R_i = I \) for all \( i \in \Omega \).


\begin{figure}[h!]
\centering
\includegraphics[width=\columnwidth]{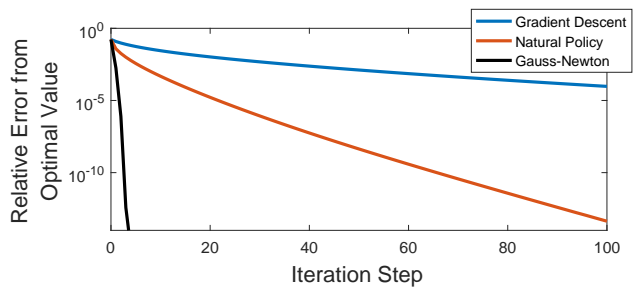}
\caption{Relative error from optimal cost for controllers computed using the proposed policy optimization methods.}
\label{fig:exact_conv}
\end{figure}

In Figure~\ref{fig:exact_conv}, we plotted the relative error of the cost function from the optimal value for all three methods.
The relative error was computed as $\left|\frac{C(K^t) - C(K^*)}{C(K^*)}\right|$.
We can see that all three methods converge to the optimal solution. As expected,
the Gauss-Newton method converges much faster than the other two methods. The step size of the natural policy gradient
method and the policy gradient method depend on various system parameters, and requires
some tuning efforts for each different problem instance.



Next, we considered a random system with 1000 states, 100 inputs, and 10 modes.
Similar results have been obtained and shown in  Figure~\ref{fig:exact_conv_1000}.
Here, despite having a larger number of states and inputs than the first system, we observe faster convergence rates due to the fact that we have a much smaller number of modes.

\begin{figure}[h!]
\centering
\includegraphics[width=\columnwidth]{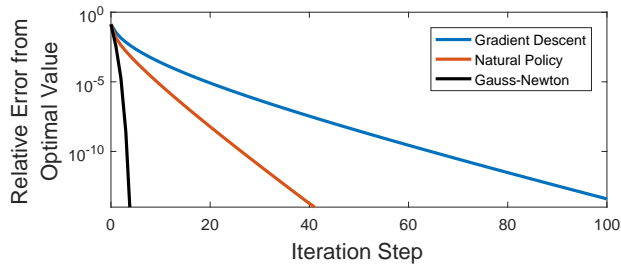}
\caption{The convergence of the Gauss-Newton, Natural policy gradient, and gradient methods on a random system with 1000 states, 100 inputs, and 10 modes.}
\label{fig:exact_conv_1000}
\end{figure}

Actually, we have run the proposed optimization methods for many other cases with different state dimensions. The observed trend is quite similar, and we omit the details of these results due to the space constraint.
Overall, the numerical results are consistent with our theory.

\section{Conclusion and Future Work}
In this paper, we have studied the global convergence of policy gradient methods for the quadratic control of Markovian jump linear systems.
First, we studied the optimization landscape of direct policy optimization for MJLS and identified a few cost properties such as coercivity, almost smoothness, and gradient dominance. 
Based on these properties, we derived global convergence guarantees for the policy gradient method, the Gauss-Newton method, and the natural policy gradient method. 
Finally, numerical results were provided to support the theoretical findings.

The policy optimization methods studied in this paper can also be implemented in a model-free manner, since various learning techniques can be used to estimate the policy gradient information from sampled trajectories. 
This will allow us to learn the optimal control of unknown MJLS without dealing with system identification.
The model-free implementations of our proposed policy optimization methods would be particularly useful for large
scale systems, where the computational complexity grows as the system size increases. 
Our theory suggests that such model-free implementations should also work as long as the gradient information is estimated with some reasonable accuracy.
An important future task is to rigorously investigate  the sample complexity of these 
data-driven policy learning methods on the MJLS LQR problem.


\bibliographystyle{IEEEtran}
\bibliography{ms}

\appendix

\numberwithin{equation}{subsection}

\subsection{Proof of the Bound \eqref{eq:q1}}

 The proof of \eqref{eq:q1} is straightforward. 
 Notice that we have
\begin{align*}
&\inner{(R+B^{\cT} \mult \E{P^K} \mult B)\mult E \mult \X^K}{E} \\
&\qquad  = \sumomega \tr{E_i^T (R_i+B_i^T\Ei{P^K} B_i)E_i \X_i^K} \\
&\qquad  \leq \sumomega  \|E_i^T(R_i+B_i^T\Ei{P^K} B_i)E_i\| \tr{\X_i^K}\\
&\qquad  \leq \sumomega \|E_i \|^2 \|R_i+B_i^T\Ei{P^K} B_i\| \tr{\X_i^K} \\
&\qquad  \leq \|E\|^2_{\max} \|R + B^{\cT} \mult \E{P^K} \mult B\|_{\max} \sumomega \tr{\X_i^K}\\
&\qquad  \leq \|E\|^2_2  \left(\|R\|_{\max} + \|B\|^2_{\max} \| P^K\|_{\max} \right) \sumomega \tr{\X_i^K}
\end{align*}
where the last step follows from \eqref{eq:Ebound}.
Hence we immediately have 
\begin{align}
\label{eq:A1}
q_1\le \left(\|R\|_{\max} + \|B\|^2_{\max} \| P^K\|_{\max} \right) \sumomega \tr{\X_i^K}.
\end{align}
Now what we need to bound $\| P^K\|_{\max}$ and $\sumomega \tr{\X_i^K}$. Recall $C(K) = \Expx{\tr{\left( \sum_{i\in \Omega} \pi_i P_i^{K} \right) x_0 x_0^T}}$.
Therefore, we have
\begin{align*}
C(K) 	&\geq \tr{ \sum_{i\in \Omega} \pi_i P_i^{K} } \sigma_{\min}\!\left({\Expx{x_0x_0^T}}\right)\\
	&\geq  \left( \sum_{i\in \Omega} \tr{P_i^{K}} \right)  \min_{i\in\Omega}(\pi_i) \sigma_{\min}\!\left({\Expx{x_0x_0^T}}\right)
\end{align*}
which leads to the following upper bound 
\begin{align}\label{eq:PKC}
    \| P^K\|_{\max}\le \sum_{i\in \Omega} \tr{P_i^{K}} \le \frac{C(K)}{\mu}.
\end{align}
Notice $\mathcal{T}$ is the adjoint operator of  $\mathcal{L}$. Hence we also have
\begin{align*}
C(K) = \inner{P^K}{X(0)} &= \inner{\sumtinf \Lt{Q+K^{\cT} \mult R \mult K}}{X(0)} \\
    &= \inner{Q+K^{\cT} \mult R \mult K}{\sumtinf \Tt{X(0)}} \\
    &= \inner{Q+K^{\cT} \mult R \mult K}{\X^K} \\
    &= \sum_{i\in\Omega} \tr{(Q_i + K_i^T R_i K_i) \X^K_i} \\
    &\geq \sum_{i\in\Omega} \sigma_{\min}(Q_i) \tr{\X^K_i} \\
    &\geq \Lambda_{\min}(Q) \sum_{i\in\Omega} \tr{\X^K_i},
\end{align*}
which leads to another useful bound 
\begin{align}\label{eq:XKC}
\sum_{i\in\Omega} \tr{\X^K_i}\le \frac{C(K)}{\Lambda_{\min}(Q)}.
\end{align} 
Substituting \eqref{eq:PKC} and \eqref{eq:XKC} into \eqref{eq:A1} leads to \eqref{eq:q1}.
\qed

\subsection{Proof of the Bound \eqref{eq:q2}}
For simplicity, we shorten the notation  $(P^K)'[E]$ as $(P^K)'$.
To prove \eqref{eq:q2},  first notice that we can use the Cauchy-Schwarz inequality to show
\begin{align*}
&\inner{B^{\cT} \mult \E{(P^K)'} \mult \varGamma \mult \X^K}{E} \\
&\qquad\qquad = \inner{E^{\cT}\mult B^{\cT}\mult \E{(P^K)'}\mult \varGamma\mult (\X^K)^{\mult 1/2}}{(\X^K)^{\mult 1/2}} \\
&\qquad\qquad \leq \|E^{\cT}\mult B^{\cT}\mult \E{(P^K)'}\mult \varGamma\mult (\X^K)^{\mult 1/2} \|_2  \|(\X^K)^{\mult 1/2}\|_2.
\end{align*}
Next, we bound $\|E^{\cT}\mult B^{\cT}\mult \E{(P^K)'}\mult \varGamma\mult (\X^K)^{\mult 1/2} \|_2$ as follows
\begin{align}\label{eq:B1}
\begin{split}
&\|E^{\cT} \mult B^{\cT} \mult \E{(P^K)'} \mult \varGamma \mult (\X^K)^{\mult 1/2} \|_2^2 \\
 &= \sumomega\tr{\Ei{(P^K)'} B_i E_i E_i^T B_i^T \Ei{(P^K)'} \varGamma_i\X^K_i \varGamma_i^T} \\
 &= \sumomega \|B_i\|^2 \|E_i\|^2 \tr{\Ei{(P^K)'} \Ei{(P^K)'} \varGamma_i\X^K_i \varGamma_i^T} \\
 &\leq \|B\|^2_{\max} \|E\|^2_{\max} \sumomega  \tr{\Ei{(P^K)'} \Ei{(P^K)'} \varGamma_i\X^K_i \varGamma_i^T} \\
 &\leq \|B\|^2_{\max} \|E\|^2_2 \sumomega  \tr{\Ei{(P^K)'} \Ei{(P^K)'} \varGamma_i\X^K_i \varGamma_i^T}
\end{split}
\end{align}

Since $\Ei{(P^K)'} \Ei{(P^K)'}$ is positive semidefinite, we have
\begin{align*}
 &\sumomega  \tr{\Ei{(P^K)'} \Ei{(P^K)'} \varGamma_i\X^K_i \varGamma_i^T}\\
   &\qquad\qquad\qquad\qquad\qquad \le \sumomega \|\Ei{(P^K)'}\|^2 \tr{ \varGamma_i\X^K_i \varGamma_i^T}\\
   &\qquad\qquad\qquad\qquad\qquad \le \|(P^K)'\|_{\max}^2 \sumomega \tr{\varGamma_i\X^K_i \varGamma_i^T}
\end{align*}
If $K\in \K$, we know $\T{\X^K} - \X^K \prec 0$ and hence the following also holds
\begin{align*}
  \sumomega \tr{\varGamma_i\X^K_i \varGamma_i^T}&=\sum_{j\in\Omega} \tr{ \sumomega p_{ij}  \varGamma_i\X^K_i \varGamma_i^T } \\
  & \le \sum_{j\in\Omega} \tr{ \mathcal{T}_j \left( \X^K \right) }\le \sum_{j\in\Omega}\tr{ \X^K_j }
\end{align*}
Therefore, substituting the above bounds into~\eqref{eq:B1} leads to
\begin{align*}
&\|E^{\cT} \mult B^{\cT} \mult \E{(P^K)'} \mult \varGamma \mult (\X^K)^{\mult 1/2} \|_2^2 \\
&\qquad\qquad\qquad\quad \leq\|B\|^2_{\max} \|E\|^2_2 \|(P^K)'\|_{\max}^2 \sum_{j\in\Omega}\tr{ \X^K_j }
\end{align*}
Since $\|(\X^K)^{\mult 1/2}\|_2^2=\sum_{j\in\Omega}\tr{ \X^K_j }$, we finally have
\begin{align}\label{eq:B2}
\begin{split}
&\inner{B^{\cT} \mult \E{(P^K)'} \mult \varGamma \mult \X^K}{E}\\ 
&\qquad\qquad \le \|B\|_{\max} \|E\|_2 \|(P^K)'\|_{\max} \sum_{j\in\Omega}\tr{ \X^K_j }
\end{split}
\end{align}
Based on \eqref{eq:B2}, proving \eqref{eq:q2} only requires showing that the following bound holds for any $\|E\|_2 = 1$ and $K\in \K$,
\begin{align}\label{eq:B3}
\|(P^K)'[E]\|_{\max} \leq  \xi \|P^K\|_{\max},
\end{align}
where $\xi$ is given as
\begin{equation*}
\xi = \frac{1}{\Lambda_{\min}(Q)} \left( \frac{1 + \|B\|_{\max}^2}{\mu} C(K) + \|R\|_{\max} \right)-1.
\end{equation*}
Once \eqref{eq:B3} is proved, it can be combined with \eqref{eq:B2}, \eqref{eq:PKC}, and \eqref{eq:XKC} to verify \eqref{eq:q2} easily.

Now the only remaining task is to prove \eqref{eq:B3}.
Let us first show $(P^K)'[E] \leq \xi P^K$ given $\|E\|_2 = 1$.
We will use Corollary 2.7 in~\cite{costa2006discrete} which states  that  $\tilde{X} \succeq X$ if $(X,\tilde{X})$ satisfy $X - \mathcal{L}(X) = S$ and $\tilde{X} - \mathcal{L}(\tilde{X}) = \tilde{S}$ with $\tilde{S} \succeq S$ and $K\in\K$.
Since $\E{P^K}\succ 0$ and $R\succ 0$, we have
\begin{align*}
&(P^K)'[E] - \mathcal{L}((P^K)'[E]) \\
=& (-B \mult E)^{\cT}\mult \E{P^K}\mult\varGamma \!+ \varGamma^{\cT} \mult \E{P^K} \mult (-B\mult E)\!  + E^{\cT} \mult R\mult K\! +\! K^{\cT} \mult R \mult E\\
 \preceq & \,\mathcal{L}(P^K) + (B\mult E)^{\cT}\mult \E{P^K} \mult B\mult E 
 +K^{\cT}\mult R\mult K + E^{\cT}\mult R\mult E \\
 = & P^K\! - Q + (B\mult E)^{\cT}\mult \E{P^K}\mult B\mult E + E^{\cT}\mult R\mult E \\
 \eqqcolon &  W
\end{align*}
If we can show $W \preceq \xi (Q+K^T R K)$, then Corollary 2.7 in~\cite{costa2006discrete} can be directly applied to show $(P^K)'[E] \leq \xi P^K$.
Note that we have the following upper bound,
\begin{align*}
&\| P^K_i + E_i^T B_i^T \Ei{P^K} B_i E_i \| \\
&\qquad\qquad\qquad\quad\; \leq  \| P^K_i \| + \|E_i^T B_i^T \Ei{P^K} B_i E_i\| \\
&\qquad\qquad\qquad\quad\; \leq \| P^K \|_{\max} + \|E\|^2_{\max} \|B\|_{\max}^2 \|P^K\|_{\max} \\
&\qquad\qquad\qquad\quad\; \leq (1 + \|B\|^2_{\max}) \frac{C(K)}{\mu}
\end{align*}
which directly leads to the following result
\begin{align*}
W &\preceq \left( (1 + \|B\|_{\max}^2) \frac{C(K)}{\mu} + \|R\|_{\max} \right) \mathcal{I} - Q \\
 &\preceq \frac{1}{\Lambda_{\min}(Q)} \left( (1 + \|B\|_{\max}^2) \frac{C(K)}{\mu} + \|R\|_{\max} \right) Q-Q\\
 &= \xi Q
\end{align*}
where $\mathcal{I} \coloneqq (I, \ldots, I)\in \M_{d\times d}^{N_s}$
Notice the bound makes sense since we know $C(K)\ge \Lambda_{\min}(Q) \mu$.
Therefore, we have $(P^K)'[E] \preceq \xi P^K$. This directly leads to \eqref{eq:B3}. Now we can complete the proof by combining \eqref{eq:B3}, \eqref{eq:B2}, \eqref{eq:PKC}, and \eqref{eq:XKC}.
\qed

\subsection{ODE Limits and Convergence}

The continuous-time ODEs are typically easier to analyze since the smoothness conditions are not required.
To gain some insight, we briefly discuss the convergence behaviors of the ODE limits of the policy gradient method, the Gauss-Newton method, and the natural policy gradient method for the MJLS LQR problem.
The ODE limit of the policy gradient method is just the so-called gradient flow:
\begin{align}\label{eq:gd_flow}
    \dot{K}(t)=-\nabla C(K).
\end{align}
The ODE limit of the Gauss-Newton method is defined as
\begin{equation}\label{eq:gn_flow}
\dot{K}(t) = -2(R + B^{\cT}\mult \E{P^K}\mult B)^{\cinv} \mult L^K 
\end{equation}
The ODE limit of the natural policy gradient method is defined as
\begin{equation}\label{eq:npg_flow}
\dot{K}(t) = -\nabla C(K)\mult (\X^{K})^{\cinv}
\end{equation}
Due to continuity, it is obvious that the above three ODEs are well posed. Based on the coercivity condition and the gradient dominance condition stated in Lemma~\ref{lemma:coercive}, we can show the following result.


\begin{thm}\label{thm:gn_flow_conv}
For $K^0\in\K$, the trajectories of the ODEs  \eqref{eq:gd_flow}, \eqref{eq:gn_flow}, and \eqref{eq:npg_flow} are all guaranteed to converge to $K^*$ exponentially. 
\end{thm}
\begin{proof}
The proof is quite standard and the details are omitted. For illustrative purposes, we present a few more steps for the gradient flow case.
Denote by $K^t$ the solution of~\eqref{eq:gd_flow}.
Define a Lyapunov function $V(K^t) \coloneqq C(K^t) - C(K^*)$.
Then we can use the gradient dominance condition to show $V(K^t) \leq \frac{1}{\alpha} \|\nabla C(K^t)\|_2^2$ with $\alpha \coloneqq 4\frac{\mu^2\Lambda_{\min}(R)}{\| \X^{K^*} \|_{\max}}$.
It immediately follows that $\dot{V}(K^t) = - \inner{\nabla C(K^t)}{\nabla C(K^t)} \leq -\alpha V(K^t)$.
Hence $V(K^t)$ converges exponentially to $0$.
\end{proof}

\end{document}